\documentclass[12pt,reqno]{amsart}
\usepackage{amsthm,amsmath,amsfonts,amssymb,euscript,hyperref,graphics,color,slashed, mathrsfs}
\usepackage{CJK}
\usepackage{graphicx}
\usepackage{makecell,rotating,multirow}
\newtheorem{theorem}{Theorem}[section]
\newtheorem{lemma}[theorem]{Lemma}

\newtheorem{remark}[theorem]{Remark}

\numberwithin{equation}{section}

\def\Lambdab{{\underline{\Lambda}}}
\def\ub{\underline{u}}

\def\C{\mathcal{C}}
\def\D{\mathcal{D}}

\def\F{\mathcal{F}}

\def\E{\mathcal{E}}
\def\Cb{\underline{\mathcal{C}}}

\def\Lb{{\underline{L}}}

\begin{document}

\title[1D Wave]{On one-dimension semi-linear wave equations with null conditions}

\author[Garving K. Luli]{Garving K. Luli}
\address{Department of Mathematics, University of California, Davis \\ USA}
\email{kluli@math.ucdavis.edu}

\author[Shiwu Yang]{Shiwu Yang}
\address{Beijing International Center for Mathematical Research, Peking University\\ Beijing, China}
\email{shiwuyang@math.pku.edu.cn}

\author[Pin Yu]{Pin Yu}
\address{Department of Mathematics and Yau Mathematical Sciences Center, Tsinghua University\\ Beijing, China}
\email{yupin@mail.tsinghua.edu.cn}

\begin{abstract}
It is well-known that in dimensions at least three semilinear wave equations with null conditions admit global solutions for small initial data. It is also known that in dimension two such result still holds for a certain class of quasi-linear wave equations with null conditions. The proofs are based on the decay mechanism of linear waves. However, in one dimension, waves do not decay. Nevertheless, we will prove that small data still lead to global solutions if the null condition is satisfied.
\end{abstract}
\maketitle
\tableofcontents

\section{Introduction}
In the past four decades, the semilinear wave equations in the following form
\begin{equation*}
\Box \varphi = Q(\partial \varphi, \partial \varphi),
\end{equation*}
have been studied intensively and have found many deep applications in geometry and physics. We assume that the field $\varphi(t,x)$ is defined on $\mathbb{R}^{n+1}$, where $x\in \mathbb{R}^n$ and $t\in \mathbb{R}$. The symbol $Q$ denotes a real valued quadratic form on $\mathbb{R}^{n+1}$. Let $\xi \in \mathbb{R}^{n+1}$ be a vector. Thus, in local frame, we have
\begin{equation*}
Q(\xi,\xi)= Q_{\mu\nu}\xi^\mu\xi^\nu.
\end{equation*}
The form $Q$ is called a \emph{null form} if for all null vectors $\xi$,  we have $ Q(\xi,\xi)= 0$. The symbol $Q(\partial \varphi, \partial \varphi)$ in the equation denotes the nonlinearity $Q^{\mu\nu}\partial_\mu\varphi\partial_\nu\varphi$. We will briefly summarize the progress on small data theory for this type of equations.

\medskip

The approach to understand the small data problem is based on the decay mechanism of linear waves. For $n\geq 4$, since linear waves decay at the rate $(1+t)^{-\frac{n-1}{2}}$(which is integrable in $t$), the small-data-global-existence type theorems hold for generic quadratic nonlinearities, see Klainerman \cite{K-80} and \cite{K-84}. However, in $\mathbb{R}^{3+1}$, the slower decay rate $(1+t)^{-1}$ just barely fails to be integrable in time, which may result in a finite time blow up of the solution even with arbitrarily small data. For example, John \cite{J-79} showed that any $C^3$ solution of the following equation
\[
\Box\varphi=|\partial_t\varphi|^2
\]
in $\mathbb{R}^{3+1}$ with nontrivial data blows up in finite time. In other words, additional conditions have to be made on the nonlinearity in order to construct a global solution.
The breakthrough along this direction was made by Klainerman in \cite{K-85} by introducing the celebrated null conditions. More precisely, if the quadratic part $Q$ of the nonlinearity is a null form, Klainerman \cite{KL-86} and Christodoulou \cite{Ch-86} have independently provided proofs for the small-data-global-existence results. Although their approaches are different, both proofs rely on the special cancelations of the null form. We remark here that in $\mathbb{R}^{3+1}$ the null condition is a sufficient but not necessary condition to obtain a small-data-global-existence result, see e.g. \cite{lindblad-weak}, \cite{igor-Msta}. For $\mathbb{R}^{2+1}$, the aforementioned classical null condition is not sufficient to guarantee a small-data-global-existence result as general cubic terms may lead to a finite time blow up of the solution. Nevertheless, Alinhac \cite{Ali-01} introduced a more restricted type of null conditions for a class of two dimensional quasilinear wave equations and under those conditions he was able to establish a small-data-global-existence result.

\medskip

All the above mentioned results are based on the following idea: the smallness of the initial data implies that the nonlinear equation can be solved for a sufficiently long time. The global solution can then be constructed once the nonlinearity decays sufficiently. The lower decay rate in low dimensions can be compensated by the special structure of the nonlinearity, namely, the null condition mentioned above. This idea may fail in $\mathbb{R}^{1+1}$ as waves in $\mathbb{R}^{1+1}$ do not decay, see e.g. \cite{john:12dwave:nonex}. Similar to $\mathbb{R}^{3+1}$, special structure of the equation may be of importance to study the asymptotic behavior of solutions of nonlinear wave equations in $\mathbb{R}^{1+1}$. Gu in \cite{gu:1dwavemap} investigated the wave map problem from $\mathbb{R}^{1+1}$ to a complete Riemannian manifold and showed that the map is regular for all time. The proof heavily relies on the geometric structure of the equations and the global regularity of the solution is a consequence of the conserved length of the tangent vector fields on the target manifold. For general nonlinear equations, Nakamura in \cite{nakamura:1dwave} studied the long time behavior of the solutions and obtained a lower bound on the life span for nonlinearity satisfying the above null condition. Moreover he was able to obtain a global solution if the nonlinear term is of the form $h(\varphi, \partial \varphi)Q^2(\partial \varphi, \partial \varphi)$, where $h$ is smooth function and $Q$ is a null form. The proof is based on an integrated local energy estimate adapted to the linear wave equation in $\mathbb{R}^{1+1}$. Such estimate plays a crucial role in the study of the asymptotic behavior of solutions of linear or nonlinear wave equations in high dimensions, see e.g. \cite{igor:redshif} and references therein. It should be noted that the integrated local energy estimate is a spacetime integral with negative weights and hence contains limited decay information on the solution. This is the reason why Nakamura requires $Q^2$ instead of $Q$ in the nonlinearity in order to obtain a global solution.

 \medskip

 The aim of the present paper is to introduce a new type of weighted energy estimates with positive weights for linear waves in $\mathbb{R}^{1+1}$. Among other things, these new estimates allow us to improve the decay estimates on the null form $Q(\partial \varphi, \partial \varphi)$. As a consequence, we strengthen the result of Nakamura in the sense that it is sufficient to require the nonlinearity to be of the form $h(\varphi, \partial \varphi)Q(\partial \varphi, \partial \varphi)$ instead of $h(\varphi, \partial \varphi)Q^2(\partial \varphi, \partial \varphi)$ in order to construct a global solution for the associated nonlinear equations.

\bigskip

We now elaborate on our main result of this paper. 

Let $\Phi(t,x):\mathbb{R}\times\mathbb{R} \rightarrow \mathbb{R}^n$ be a vector valued function. More explicitly, we can write $\Phi(t,x)=\big(\Phi_1(t,x),\cdots,\Phi_n(t,x)\big)$. 

We consider the following system of wave equations:
\begin{equation}\label{main equation}
\begin{split}
\Box\Phi &= N(t,x),\\
(\Phi,\partial_t \Phi)\big|_{t=0}&= \varepsilon(F(x),G(x)).
\end{split}
\end{equation}
Here $\varepsilon \geq 0$ is a constant and $F, G$ are smooth real-valued functions; $N(t,x) = \big(N_1(t,x),\cdots,N_n(t,x)\big)$ are quadratic nonlinearities in $\partial \Phi$ with null conditions. More precisely, for $i,k,l=1,2,\cdots, n$, there exist constants $C^{kl}_i$ so that $N_i(t,x)$ can be written as
\begin{equation*}
\begin{split}
N_{i}(t,x) &=\sum_{k,l=1}^n C^{kl}_{i}\cdot (\partial_t+\partial_x)\Phi_k \cdot (\partial_t-\partial_x)\Phi_l\\
&=\sum_{k,l=1}^n C^{kl}_{i}\cdot L\Phi_k \cdot \Lb\Phi_l,
\end{split}
\end{equation*}
where $L=\partial_t+\partial_x$ and $\Lb=\partial_t-\partial_x$ are the two principal null vectors. Thus, the quadratic nonlinearities $N_{i}$'s are linear combinations of the quadratic forms of the following type:
\begin{equation*}
Q(\partial \Phi_k,\partial \Phi_l) = \alpha L \Phi_k \Lb \Phi_l+\beta \Lb \Phi_k L \Phi_l,
\end{equation*}
where $\alpha$, $\beta$ are real numbers. In particular, this means that in the frame $(L,\Lb)$, as a matrix, $Q$ can be written as $
  \left( {\begin{array}{cc}
   0 & \alpha \\
   \beta & 0 \\
  \end{array} } \right)$. This means $Q(L,L)=0$ and $Q(\Lb,\Lb)=0$. As a conclusion, $Q$ is a null form on $\mathbb{R}^{1+1}$. On the other hand, it is obvious that a null form must be of this form. Therefore, the system \eqref{main equation} represents all semilinear wave equations with quadratic null form nonlinearities. Schematically, we also write \eqref{main equation} as
\begin{equation*}
\begin{split}
\Box\Phi &= Q(\partial\Phi,\partial \Phi),
\end{split}
\end{equation*}
to emphasize that $Q$ is a null form.

\medskip

We remark that this system of equations indeed can be viewed as a model problem for incompressible MHD systems placed in a strong magnetic background. Thus, we can use \eqref{main equation} to describe the propagation of Alfv\'en waves and we refer to \cite{HXY} for more details.

\bigskip

The main theorem of the paper is as follows
\begin{theorem} In the setting of system \eqref{main equation}, we have the following: For all $0<\delta<1$, there exists a universal small constant $\varepsilon_0 > 0$ such that the following holds.  Suppose
\begin{equation*}
\sum_{k=0,1}\int_{\mathbb{R}}(1+|x|)^{2+2\delta}\big(|\partial_x^k \partial_x F|^2+|\partial_x^k G|^2\big)dx \leq 1.
\end{equation*}
Then for all positive constants $\varepsilon<\varepsilon_0$, system \eqref{main equation} admits global solutions.
\end{theorem}

In other words, as long as the functions $F$ and $G$ appearing in the initial data of \eqref{main equation} have suitable decay at infinity as measured by the weighted Sobolev norm, we can construct a global solution for the system \eqref{main equation}.

We briefly discuss the key idea behind the proof. As we have mentioned above, in $\mathbb{R}^{1+1}$, there is no decay for linear waves. However, we claim that although the solution $\Phi$ does not decay, the nonlinearity $N(t,x)$ does decay. This is the key observation that allows us to prove the small data global existence result. The geometric interpretation for this is the following: If we think of the solution behaves as linear waves, then we may regard $(\partial_t+\partial_x)\Phi_k $ as left-traveling waves and $(\partial_t-\partial_x)\Phi_l$ as right-traveling waves. For sufficiently long time, which is ensured by the smallness of the initial data, these two families of waves will be separated in space. On the other hand, the null conditions can be phrased as left-traveling waves coupled only with right-traveling waves, since they are far away from each other for large time, the spatial decay now yields decay in time. This new decay mechanism is strongly in contrast with that in the higher dimensional cases, where the improved decay comes from the tangential derivative of the waves along outgoing light cones.

\textbf{Acknowledgments} The first author is partially supported by NSF grant DMS-1554733. The second author is partially supported by NSFC-11701017. The third author is grateful to UC Davis for the support through the New Research Initiatives and Collaborative Interdisciplinary Research Grant during his visit.

\section{Preliminaries: the geometry of $\mathbb{R}^{1+1}$ and linear estimates}

On the two dimensional Minkowski spacetime $\mathbb{R}^{1+1}$, we will use two coordinate systems: the standard Cartesian coordinates $(t,x)$ and the null coordinates $(u,\ub)$. The coordinate functions in the null coordinates are the standard optical functions defined as follows
\begin{equation*}
u=\frac{1}{2}(t-x), \ \ \ub=\frac{1}{2}(t+x).
\end{equation*}
We use $g_{\alpha\beta}$ to denote the standard Minkowski metric on $\mathbb{R}^{1+1}$. In other words, the metric can be written down explicitly in the Cartesian coordinates as
\begin{equation*}
g= -dt^2+dx^2.
\end{equation*}
In the null coordinates, we have
\begin{equation*}
g= -2du d\ub.
\end{equation*}
We have two globally defined null vector fields
\begin{equation*}
L=\partial_t+\partial_x, \ \ \Lb=\partial_t-\partial_x.
\end{equation*}
The metric $g$ can be expressed by the null frame as follows
\begin{equation}\label{null}
g(L,L)=g(\Lb,\Lb)=0, \ \  g(L,\Lb)=-2.
\end{equation}
By definition, we also have
\begin{equation}\label{vectorfield}
Lu=0, \ \ L \ub =1, \ \ \Lb\ub = 0, \ \ \Lb u=1.
\end{equation}

\medskip

We use $\Sigma_{t_0}$ to denote the following time slice in $\mathbb{R}^{1+1}$:
\begin{equation*}
\Sigma_{t_0} :=\big\{(t,x) \,\big|\, t= t_0\big\}.
\end{equation*}
We write$\D_{t_0}$ to denote the following spacetime region:
\begin{equation*}
\D_{t_0} :=\big\{(t,x) \,\big|\, 0\leq t \leq t_0\big\}.
\end{equation*}
In other words, we have $\displaystyle \D_{t_0} = \bigcup_{0\leq t\leq t_0} \Sigma_t$.

The level sets of the optical functions $u$ and $\ub$ define two global null foliations of $ \D_{t_0}$. More precisely, given $t_0 >0$, $u_0$ and $\ub_0$, we define the right-going null curve segment $\C_{u_0}^{t_0}$ as
\begin{equation*}
\C_{u_0}^{t_0} :=\big\{(t,x) \,\big|\, u=\frac{t-x}{2}=u_0, 0\leq t \leq t_0\big\},
\end{equation*}
and the left-going null curve segment $\Cb_{\ub_0}^{t_0}$ as
\begin{equation*}
\Cb_{\ub_0}^{t_0} :=\big\{(t,x) \,\big|\, \ub =\frac{t+x}{2}=\ub_0, 0\leq t\leq t_0\big\}.
\end{equation*}
We also define spatial segments
\begin{equation*}
\Sigma^+_{t_0,u_0} :=\big\{(t,x) \,\big|\, t= t_0, x \geq t_0-2u_0 \big\},
\end{equation*}
and
\begin{equation*}
\Sigma^-_{t_0,\ub_0} :=\big\{(t,x) \,\big|\, t= t_0, x \leq 2\ub_0-t_0 \big\}.
\end{equation*}
Similarly, we define spacetime regions
\begin{equation*}
\D^+_{t_0,u_0} :=\big\{(t,x) \,\big|\, 0 \leq t \leq t_0, x \geq t-2u_0 \big\},
\end{equation*}
and
\begin{equation*}
\D^-_{t_0,\ub_0} :=\big\{(t,x) \,\big|\, 0\leq  t \leq t_0, x \leq 2\ub_0-t \big\}.
\end{equation*}
We depict the above geometric constructions in the following the following picture:

\ \ \ \ \ \ \ \ \ \ \ \ \  \includegraphics[width=5in]{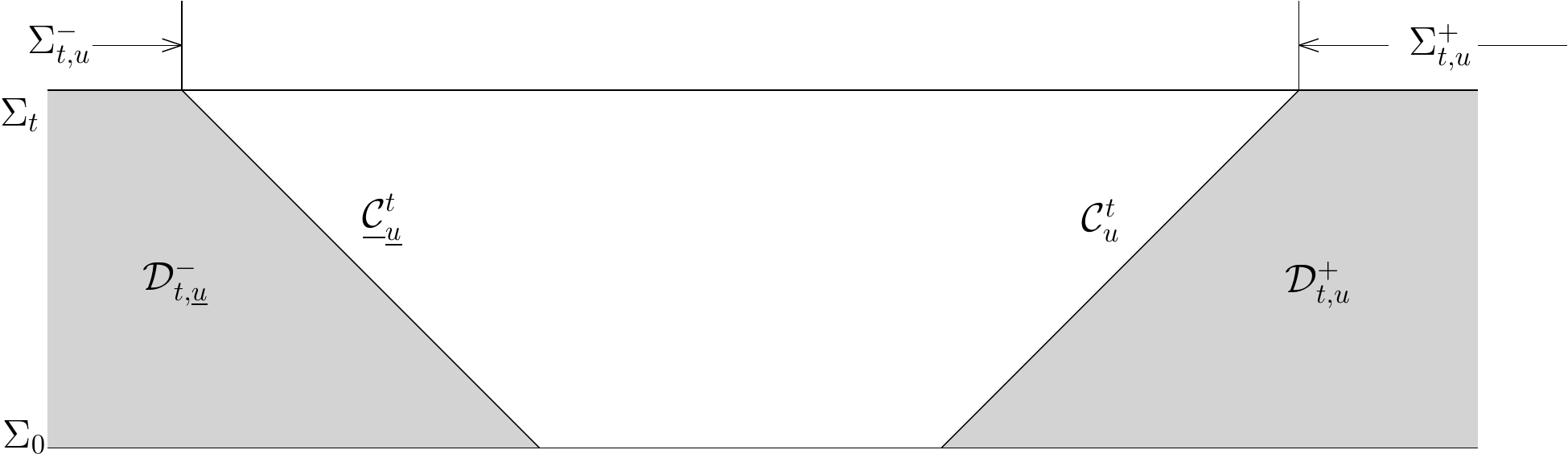}

\noindent The grey regions are $\D^+_{t,u}$ and $\D^-_{t,\ub}$. The entire region enclosed by $\Sigma_t$ and $\Sigma_0$ is $\D_t$.

\bigskip

Let $Z$ be a smooth vector field defined on $\mathbb{R}^{1+1}$. We recall that its deformation tensor $\,^{(Z)}\pi_{\mu\nu}$ is defined as $\displaystyle \frac{1}{2}\mathcal{L}_Z g$ (the Lie derivative) where $g$ is the Minkowski metric. In other words, the deformation tensor of $Z$ is a two tensor and its components are given by
\begin{equation*}
\,^{(Z)}\pi_{\mu\nu}=\frac{1}{2}(\nabla_\mu Z_\nu + \nabla_\nu Z_\mu),
\end{equation*}
where $\nabla$ is the Levi-Civita connection of $g$. If there exists a function $\Omega$ so that $\,^{(Z)}\pi_{\mu\nu} = \Omega \cdot g_{\alpha \beta}$, we say that $Z$ is a conformal Killing vector field. Indeed, for a conformal Killing vector field, its one parameter subgroup of the diffeomorphisms consists of conformal transformations of $(\mathbb{R}^{1+1},g)$.
\begin{lemma}\label{conformal vector fields}
Let $\Lambda$ and $\Lambdab$ be two smooth $\mathbb{R}$-valued one variable functions. Then the vector field
\begin{equation}\label{zconformal}
Z =\Lambdab(\ub)L+\Lambda(u)\Lb
\end{equation}
is a conformal Killing vector field on $\mathbb{R}^{1+1}$. Moreover, we have
\begin{equation}\label{zconformald}
\,^{(Z)}\pi_{\mu\nu} =\frac{1}{2}\big(\Lambdab'(\ub)+\Lambda'(u)\big)g_{\mu\nu}.
\end{equation}
\end{lemma}
\begin{proof}
It suffices to check \eqref{zconformald}. By linearity and symmetry, it suffices to show that
\begin{equation*}
\,^{(Z)}\pi_{\mu\nu} =\frac{1}{2}\Lambda'(u)\cdot g_{\mu\nu},
\end{equation*}
for $Z =\Lambda(u)\Lb$. Indeed, by \eqref{null} and \eqref{vectorfield}, we have
\begin{align*}
\,^{(Z)}\pi_{LL}&=g\big(\nabla_L (\Lambda(u)\Lb),L\big)=0,\ \
\,^{(Z)}\pi_{\Lb\Lb}=g\big(\nabla_\Lb (\Lambda(u)\Lb),\Lb\big)=0,\\
\,^{(Z)}\pi_{L\Lb}&=\frac{1}{2}g\big(\nabla_L (\Lambda(u)\Lb),\Lb\big)+\frac{1}{2}g\big(\nabla_\Lb (\Lambda(u)\Lb),L\big)=-\Lambda'(u).
\end{align*}
This proves the lemma.
\end{proof}

We consider a solution $\varphi(t,x)$ to the following scalar linear wave equation on $\D_t$:
\begin{equation*}
\Box\varphi = \rho.
\end{equation*}
The energy-momentum tensor associated to $\varphi$ is defined as
\begin{equation*}
T_{\mu\nu}=\nabla_\mu \varphi \nabla_\nu \varphi -\frac{1}{2}g_{\mu\nu}\nabla^\alpha\varphi \nabla_\alpha \varphi.
\end{equation*}
It is straightforward to see that $T(L,L)=|L\varphi|^2$, $T(\Lb,\Lb)=|\Lb\varphi|^2$ and $T(L,\Lb)=0$.
We can also compute the divergence of $T_{\mu\nu}$:
\begin{equation*}
\nabla^\nu T_{\mu\nu}= \rho \nabla_\mu \varphi.
\end{equation*}
\begin{remark}[Conformal property] On $\mathbb{R}^{1+1}$, the theory of linear wave equations is a conformal theory, i.e., the associated deformation tensor $T_{\mu\nu}$ is trace-free:
\begin{equation*}
g^{\mu\nu}T_{\mu\nu}=0.
\end{equation*}
\end{remark}

Let $Z$ be a smooth (multiplier) vector field. Moreover, we assume that $Z$ is a conformal Killing vector field. Therefore, the current
\begin{equation*}
^{(Z)}J_\mu= T_{\mu\nu}Z^\nu,
\end{equation*}
satisfies the following divergence identity:
\begin{equation}\label{divergence identity}
\nabla^\mu \,^{(Z)}J_\mu=\rho Z(\varphi).
\end{equation}
This can be proved by the following computation:
\begin{align*}
\nabla^\mu \,^{(Z)}J_\mu&=\nabla^\mu \big(T_{\mu\nu}Z^\nu\big)=\rho Z(\varphi)+\underbrace{T_{\mu\nu}\nabla^\mu Z^\nu}_{=T_{\mu\nu}\,^{(Z)}\pi^{\mu\nu}\sim T_{\mu\nu} g^{\mu\nu}}.
\end{align*}
We used the fact that $Z$ is a conformal Killing vector field in the last step. Since $T_{\mu\nu}$ is tracefree, the last term vanishes and this proves \eqref{divergence identity}. In applications, we will always take $Z =\Lambdab(\ub)L+\Lambda(u)\Lb$ as in Lemma \ref{conformal vector fields}.

\bigskip

We will integrate the divergence identity \eqref{divergence identity} in the domain $\mathcal{D}_{t,\ub}^-$ and the domain is depicted as follows:

\ \ \ \ \ \ \ \ \ \ \ \ \ \ \ \ \ \ \ \ \ \ \ \ \ \ \ \ \ \ \ \ \ \ \ \ \ \ \  \includegraphics[width=3.2in]{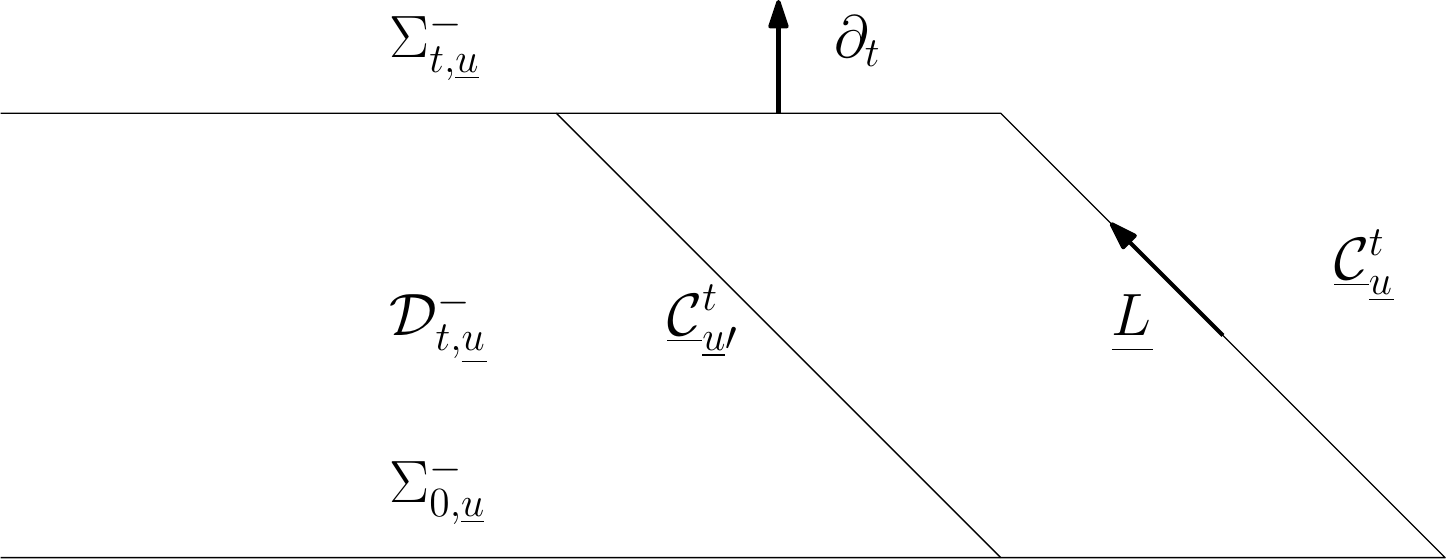}

\noindent We remark that the domain $\mathcal{D}_{t,\ub}^-$ is foliated by $\Cb^t_{\ub'}$ for $\ub'\leq \ub$. Since the (Lorentian) normal of $\Cb_{\ub}^t$ is $\Lb$, the Stokes formula implies that
\begin{equation*}
\int_{\Sigma_{t,\ub}^-}T(\partial_t,Z)+ \int_{\Cb_{\ub}^t} T(\Lb,Z)=\int_{\Sigma_{0,\ub}^-}T(\partial_t,Z)+\int\!\!\!\int_{\D^{-}_{t,\ub}}\rho Z\varphi.
\end{equation*}
We take $Z=\Lambda(u)\Lb$. In view of the fact that $T(L,\Lb)=0$, we obtain
\begin{equation*}
\frac{1}{2}\int_{\Sigma_{t,\ub}^-}\Lambda(u) (\Lb\varphi)^2+ \int_{\Cb_{\ub}^t}\Lambda(u) |\Lb\varphi|^2=\frac{1}{2}\int_{\Sigma_{0,\ub}^-}\Lambda(u) (\Lb\varphi)^2+\int\!\!\!\int_{\D^{-}_{t,\ub}} \Lambda(u)\Lb\varphi\cdot\rho
\end{equation*}
From now on, we will assume that the weight function $\Lambda \geq 0$. By enlarging the domains from $\Sigma^-_{0,\ub}$ to $\Sigma_{0}$ and from $\D^{-}_{t,\ub}$ to $\D_{t}$, we obtain that
\begin{equation*}
\int_{\Sigma_{t,\ub}^-}\Lambda(u) (\Lb\varphi)^2+ \int_{\Cb_{\ub}^t}\Lambda(u) |\Lb\varphi|^2 \lesssim \int_{\Sigma_{0}}\Lambda(u) (\Lb\varphi)^2+\int\!\!\!\int_{\D_{t}} \Lambda(u)|\Lb\varphi||\rho|.
\end{equation*}
Here and in the sequel, the notation $A\lesssim B$ means that there is a universal constant $C$ such that $A\leq CB$. Since this inequality holds for all $\ub$, we obtain that
\begin{equation}\label{lambda1}
\int_{\Sigma_{t}}\Lambda(u) (\Lb\varphi)^2+ \sup_{\ub\in \mathbb{R}}\int_{\Cb_{\ub}^t}\Lambda(u) |\Lb\varphi|^2 \lesssim \int_{\Sigma_{0}}\Lambda(u) (\Lb\varphi)^2+\int\!\!\!\int_{\D_{t}} \Lambda(u)|\Lb\varphi||\rho|.
\end{equation}
Similarly, we can work on $\D^{+}_{t,u}$ by using the multiplier $Z=\Lambdab(\ub)L$ and this yields
\begin{equation}\label{lambda2}
\int_{\Sigma_{t}}\Lambdab(\ub) (L\varphi)^2+ \sup_{u\in \mathbb{R}}\int_{\C_{u}^t}\Lambdab(\ub) |L\varphi|^2 \lesssim \int_{\Sigma_{0}}\Lambdab(\ub) (L\varphi)^2+\int\!\!\!\int_{\D_{t}} \Lambdab(\ub)|L\varphi||\rho|.
\end{equation}
By adding estimates \eqref{lambda1} and \eqref{lambda2} together, we finally obtain the energy estimates for linear equations: For $\Lambda \geq 0$, we have
\begin{equation}\label{linear energy estimates}
\begin{split}
& \ \ \int_{\Sigma_{t}}\Lambda(u) (\Lb\varphi)^2+\Lambdab(\ub) (L\varphi)^2+\sup_{\ub\in \mathbb{R}}\int_{\Cb_{\ub}^t}\Lambda(u) |\Lb\varphi|^2 +\sup_{u\in \mathbb{R}}\int_{\C_{u}^t}\Lambdab(\ub) |L\varphi|^2\\
& \leq C_0\Big(\int_{\Sigma_{0}}\Lambda(u) (\Lb\varphi)^2+\Lambdab(\ub) (L\varphi)^2+\int\!\!\!\int_{\D_{t}} \big(\Lambda(u)|\Lb\varphi|+\Lambdab(\ub)|L\varphi|\big)|\rho|\Big),
\end{split}
\end{equation}
where we can take $C_0$ to be at least $1$.

\section{The proof of the main theorem}

For the sake of simplicity, we will stick to the schematic form $\Box \Phi = Q(\partial \Phi, \partial \Phi)$ of the main equation \eqref{main equation} by ignoring all the constants and indices in \eqref{main equation}. In such a way, we may think of the system as a single equation for scalar functions. On the other hand, we can replace $\Phi$ in the following proof by $\Phi_k$ and then sum over $k$ to complete the proof for the original system.

\begin{lemma}\label{lemma commute derivatives} For all $k\leq 1$, we have
\begin{equation*}
\Box \partial^k_x\Phi = \sum_{i+j=k}Q(\partial \partial_x^i \Phi,\partial \partial_x^j \Phi),
\end{equation*}
where the $Q$'s are null forms and we have omitted all the irrelavent constants in front of $Q$.
\end{lemma}
\begin{proof}
We simply commute $\partial^k$ with $\Box \Phi = Q(\partial \Phi, \partial \Phi)$. The proof is straightforward.
\end{proof}

In the rest of the paper, we choose $\Lambda(u)$ and $\Lambdab(\ub)$ as follows:
\begin{equation}\label{lambdafunction}
\Lambda(u) = (1+|u|^2)^{1+\delta}, \ \ \Lambdab(\ub)=(1+|\ub|^2)^{1+\delta}.
\end{equation}

In view of \eqref{linear energy estimates}, for $k=0,1$, we define energy norms as follows:
\begin{equation*}
\begin{split}
\E_k(t)& = \int_{\Sigma_{t}}\Lambda(u) (\Lb\partial^k_x \Phi)^2+\Lambdab(\ub) (L\partial^k_x \Phi)^2,\\
\F_k(t)& = \sup_{\ub\in \mathbb{R}}\int_{\Cb_{\ub}^t}\Lambda(u) |\Lb\partial^k_x\Phi|^2 +\sup_{u\in \mathbb{R}}\int_{\C_{u}^t}\Lambdab(\ub) |L\partial^k_x\Phi|^2.
\end{split}
\end{equation*}
We also define the total energy norms as follows:
\begin{equation*}
\E(t) = \sum_{k=0}^1 \E_k(t),\ \ \F(t) = \sum_{k=0}^1 \F_k(t).
\end{equation*}
We notice that if $t=0$, we have $\F(0)=0$. The data determine $\E(0)$. Indeed, the functions $F$ and $G$ determine a constant $C_1$ so that
\begin{equation}\label{initial energy}
\E(0) = C_1 \varepsilon^2.
\end{equation}
Now we use the method of continuity: We assume that the solution $\Phi$ exists for $t\in [0,T^*]$ so that it has the following bound
\begin{equation}\label{bootstrap assumption}
\E(t)+\F(t) \leq 6 C_0 C_1 \varepsilon^2.
\end{equation}
Since this bound holds for $t=0$, we can always find such a $T^*$. In the rest of the paper, we will show that under assumption \eqref{bootstrap assumption}, we can indeed prove a better bound. Namely, for all $t\in[0,T^*]$, we will show that there exists a universal constant $\varepsilon_0$, so that we have the improved estimate:
\begin{equation}\label{bootstrap assumption improved}
\E(t)+\F(t) \leq 4 C_0 C_1 \varepsilon^2
\end{equation}
for all $\varepsilon <\varepsilon_0$.
The constant $\varepsilon_0$ will be independent of $T_*$. Thus assumption \eqref{bootstrap assumption} will never be saturated so that we can always continue $T_*$ to $+\infty$. This will prove the global existence for small data solutions of \eqref{main equation}. Therefore, the crux of the matter boils down to proving \eqref{bootstrap assumption improved} under \eqref{bootstrap assumption}.

\medskip

In view of \eqref{bootstrap assumption}, we first have the following pointwise bounds:
\begin{lemma}\label{lemma pointwise bound} Under assumption \eqref{bootstrap assumption}, there exists a universal constant $C_2$ so that
\begin{equation*}
|L\Phi(t,x)|\leq \frac{C_2\varepsilon}{\Lambdab(\ub)^\frac{1}{2}}, \ \ |\Lb\Phi(t,x)|\leq \frac{C_2\varepsilon}{\Lambda(u)^\frac{1}{2}}.
\end{equation*}
\end{lemma}
\begin{proof}
It suffices to prove the first inequality. The second can be proved in exactly the same way. Indeed, according to the Sobolev inequality on $\mathbb{R}$, since $|\partial_x\Lambdab(\ub)^\frac{1}{2}|\leq \Lambdab(\ub)^\frac{1}{2}$, we have
\begin{align*}
|\Lambdab(\ub)^\frac{1}{2}L \Phi|^2 &\lesssim \|\Lambdab(\ub)^\frac{1}{2}L \Phi\|^2_{L^2(\mathbb{R}_x)}+\|\partial_x\big(\Lambdab(\ub)^\frac{1}{2}L \Phi\big)\|^2_{L^2(\mathbb{R}_x)}\\
&\lesssim \|\Lambdab(\ub)^\frac{1}{2}L \Phi\|^2_{L^2(\mathbb{R}_x)}+\|\Lambdab(\ub)^\frac{1}{2} L \partial_x  \Phi\|^2_{L^2(\mathbb{R}_x)} \\
&\lesssim \E(t) \\
&\leq 6C_0C_1 \varepsilon^2.
\end{align*}
The  desired inequality follows immediately.
\end{proof}
\begin{remark}
The lemma shows $|L\Phi(t,x)|\lesssim \frac{\varepsilon}{(1+|t+x|)^{1+\delta}}$ and $|\Lb\Phi(t,x)|\leq \frac{\varepsilon}{(1+|t+x|)^{1+\delta}}$. Thus, although the waves have no decay in time, they still decay spatially away from their centers.
\end{remark}
Similar to the above lemma, we also have
\begin{lemma}\label{lemma pointwise bound different weights} Under assumption \eqref{bootstrap assumption}, there exists a universal constant $C_3$ so that
\begin{equation*}
\|\frac{\Lambda(u)^\frac{1}{2}}{\Lambdab(\ub)^\frac{1}{4}}\Lb\Phi(t,x)\|_{L^\infty(\Sigma_t)}\leq C_3\Big(\|\frac{\Lambda(u)^\frac{1}{2}}{\Lambdab(\ub)^\frac{1}{4}}\Lb\Phi(t,x)\|_{L^2(\Sigma_t)}+\|\frac{\Lambda(u)^\frac{1}{2}}{\Lambdab(\ub)^\frac{1}{4}}\Lb\partial_x\Phi(t,x)\|_{L^2(\Sigma_t)}\Big),
\end{equation*}
and
\begin{equation*}
\|\frac{\Lambdab(\ub)^\frac{1}{2}}{\Lambda(u)^\frac{1}{4}}L\Phi(t,x)\|_{L^\infty(\Sigma_t)}\leq C_3\Big(\|\frac{\Lambdab(\ub)^\frac{1}{2}}{\Lambda(u)^\frac{1}{4}}L\Phi(t,x)\|_{L^2(\Sigma_t)}+\|\frac{\Lambdab(\ub)^\frac{1}{2}}{\Lambda(u)^\frac{1}{4}}L\partial_x\Phi(t,x)\|_{L^2(\Sigma_t)}\Big).
\end{equation*}
\end{lemma}
\begin{proof}
By using the new weight function $\frac{\Lambda(u)^\frac{1}{2}}{\Lambdab(\ub)^\frac{1}{4}}$, this inequality can be derived in exactly the same manner as in Lemma \ref{lemma pointwise bound}. In fact, it suffices to notice that
\begin{align*}
\partial_x\Big(\frac{\Lambda(u)^\frac{1}{2}}{\Lambdab(\ub)^\frac{1}{4}}\Big)\lesssim \frac{\Lambda(u)^\frac{1}{2}}{\Lambdab(\ub)^\frac{1}{4}}.
\end{align*}
This can be checked by a direct computation.
\end{proof}
\medskip

For each $k=0,1$, we take $\varphi = \partial_x^k \Phi$ in the linear energy estimates \eqref{linear energy estimates}. In view of Lemma \ref{lemma commute derivatives}, we obtain
\begin{equation*}
\begin{split}
 \E_k(t)+\F_k(t)\leq C_0\Big(\E_k(0)+\sum_{i+j= k}\int\!\!\!\int_{\D_{t}} \big(\Lambda(u)|\Lb\partial_x^k\Phi|+\Lambdab(\ub)|L\partial_x^k\Phi|\big)\big|Q(\partial \partial_x^i \Phi,\partial \partial_x^j \Phi)\big|\Big).
\end{split}
\end{equation*}
We can also take the sum of all the above estimates and this yields
\begin{equation}
\label{energy estimates}
\begin{split}
 \E(t)+\F(t)&\leq 2C_0C_1\varepsilon^2+C_0\sum_{i+j= k, \atop 0\leq k \leq 1}\int\!\!\!\int_{\D_{t}} \big(\Lambda(u)|\Lb\partial_x^k\Phi|+\Lambdab(\ub)|L\partial_x^k\Phi|\big)\big|Q(\partial \partial_x^i \Phi,\partial \partial_x^j \Phi)\big|\\
 &\leq 2C_0C_1\varepsilon^2+C_0\underbrace{\sum_{i+j= k, \atop 0\leq k \leq 1}\int\!\!\!\int_{\D_{t}} \Lambda(u)|\Lb\partial_x^k\Phi|\big|Q(\partial \partial_x^i \Phi,\partial \partial_x^j \Phi)\big|}_{\mathbf{I}}\\
 &\ \ \ \ \ \ \ \ \ \ \ \ \ \ \ \ +C_0\sum_{i+j= k, \atop 0\leq k \leq 1}\int\!\!\!\int_{\D_{t}}\Lambdab(\ub)|L\partial_x^k\Phi|\big|Q(\partial \partial_x^i \Phi,\partial \partial_x^j \Phi)\big|.
\end{split}
\end{equation}
To bound the nonlinear terms, in view of the symmetry, it suffices to bound the term $\mathbf{I}$ in \eqref{energy estimates}. Since $Q$ is a null form, we have
\begin{equation*}
|Q(\partial \partial_x^i \Phi,\partial \partial_x^j \Phi)\big|\lesssim |L\partial_x^i \Phi||\Lb\partial_x^j \Phi|+|\Lb\partial_x^i \Phi||L\partial_x^j \Phi|.
\end{equation*}
Therefore, we may rewrite $\mathbf{I}$ as
\begin{equation}\label{first bound on I}
\mathbf{I}\lesssim \sum_{i+j= k, \atop 0\leq k \leq 1}\int\!\!\!\int_{\D_{t}} \Lambda(u)|\Lb\partial_x^k\Phi||L\partial_x^i \Phi||\Lb\partial_x^j \Phi|.
\end{equation}
We may classify the terms in the integrand into two cases according to $i=0$ or $i=1$.

\medskip

\noindent {\bf Case 1: $i=0$}. We have to bound the term
\begin{align*}
\mathbf{I_1}=\int\!\!\!\int_{\D_{t}} \Lambda(u)|\Lb\partial_x^k\Phi||L  \Phi||\Lb\partial_x^j \Phi|.
\end{align*}
 Applying Lemma \ref{lemma pointwise bound} to $|L\Phi|$, we have	
\begin{align*}
\mathbf{I_1}&\lesssim \int\!\!\!\int_{\D_{t}} \frac{\varepsilon}{\Lambdab(\ub)^\frac{1}{2}}\Lambda(u)|\Lb\partial_x^k\Phi||\Lb\partial_x^j \Phi|\\
&\lesssim \int\!\!\!\int_{\D_{t}} \frac{\varepsilon}{\Lambdab(\ub)^\frac{1}{2}}\Big(\Lambda(u)|\Lb\partial_x^k\Phi|^2+ \Lambda(u)|\Lb\partial_x^j \Phi|^2\Big).
\end{align*}
Since the spacetime region $\D_{t}$ is foliated by $\Cb_{\ub}^t$ for $\ub\in \mathbb{R}$, by Fubini's theorem, we have
\begin{align*}
\mathbf{I_1}&\lesssim \int_\mathbb{R}\Big[\int_{\Cb_{\ub}^t} \frac{\varepsilon}{\Lambdab(\ub)^\frac{1}{2}}\Big(\Lambda(u)|\Lb\partial_x^k\Phi|^2+ \Lambda(u)|\Lb\partial_x^j \Phi|^2\Big)\Big]d\ub\\
&= \int_\mathbb{R}\frac{\varepsilon}{\Lambdab(\ub)^\frac{1}{2}}\Big(\underbrace{\int_{\Cb_{\ub}^t} \Lambda(u)|\Lb\partial_x^k\Phi|^2+ \Lambda(u)|\Lb\partial_x^j \Phi|^2}_{\leq \F(t)}\Big)d\ub\\
&\lesssim \int_\mathbb{R}\frac{\varepsilon^3}{\Lambdab(\ub)^\frac{1}{2}}d\ub.
\end{align*}
On the other hand, we know that $\Lambdab(\ub)^\frac{1}{2}\sim (1+|\ub|)^{1+\delta}$ for $|\ub|\rightarrow \infty$, thus the above integral is finite. As a consequence, we obtain that
\begin{align*}
\mathbf{I_1}\lesssim \varepsilon^3.
\end{align*}

\noindent {\bf Case 2: $i=1$}. In this case, we must have $j=0$ and $k=1$, since $i+j =k \leq 1$. We now have to bound the term
\begin{align*}
\mathbf{I_1}&=\int\!\!\!\int_{\D_{t}} \Lambda(u)|\Lb\partial_x\Phi||L \partial_x \Phi||\Lb \Phi|\\
&=\int\!\!\!\int_{\D_{t}} \underbrace{\Big(\Lambdab(\ub)^{-\frac{1}{4}}\Lambda(u)^\frac{1}{2}|\Lb\partial_x\Phi|\Big)}_{L^2_tL^2_x}\underbrace{\Big(\Lambdab(\ub)^\frac{1}{2}|L\partial_x \Phi|\Big)}_{L^\infty_tL^2_x}\underbrace{\Big(\Lambdab(\ub)^{-\frac{1}{4}}\Lambda(u)^\frac{1}{2}|\Lb \Phi|\Big)}_{L^2_tL^\infty_x}\\
&\leq \underbrace{\Big(\int\!\!\!\int_{\D_{t}} \frac{\Lambda(u)|\Lb\partial_x\Phi|^2}{\Lambdab(\ub)^\frac{1}{2}}\Big)^\frac{1}{2}}_{\mathbf{I_2}}\underbrace{\sup_{t\in[0,T^*]}\Big(\int_{\Sigma_{t}} \Lambdab(\ub)|L\partial_x\Phi|^2\Big)^\frac{1}{2}}_{\mathbf{I_3}}\underbrace{\Big(\int_{0}^t \|\frac{\Lambda(u)^\frac{1}{2}}{\Lambdab(\ub)^{\frac{1}{4}}}|\Lb \Phi|\|^2_{L^\infty(\Sigma_\tau)}d\tau\Big)^\frac{1}{2}}_{\mathbf{I_4}}.
\end{align*}
The term $\mathbf{I_2}$ can be treated in a similar manner as for $\mathbf{I_1}$:
\begin{align*}
(\mathbf{I_2})^2 &\leq \int_\mathbb{R}\int_{\Cb_{\ub}^t} \frac{1}{\Lambdab(\ub)^\frac{1}{2}} \Lambda(u)|\Lb\partial_x\Phi|^2 d\ub\\
&= \int_\mathbb{R}\frac{1}{\Lambdab(\ub)^\frac{1}{2}}\Big(\underbrace{\int_{\Cb_{\ub}^t} \Lambda(u)|\Lb\partial_x\Phi|^2}_{\leq \F_1(t)}\Big)d\ub\\
&\lesssim \int_\mathbb{R}\frac{\varepsilon^2}{\Lambdab(\ub)^\frac{1}{2}}d\ub.
\end{align*}
Since $\Lambdab(\ub)^\frac{1}{2}$ is integrable in $\ub$, we have
\begin{align*}
I_2 \lesssim \varepsilon.
\end{align*}
The term $\mathbf{I_3}$ is part of the energy norm $\E_1(t)$, thus, by \eqref{bootstrap assumption}, we have
\begin{align*}
I_3 \lesssim \varepsilon.
\end{align*}
For the term $\mathbf{I_4}$, according to Lemma \ref{lemma pointwise bound different weights}, we have
\begin{align*}
\mathbf{I_4} &\lesssim
\Big(\int_{0}^t \|\frac{\Lambda(u)^\frac{1}{2}}{\Lambdab(\ub)^\frac{1}{4}}\Lb\Phi(t,x)\|^2_{L^2(\Sigma_\tau)}+\|\frac{\Lambda(u)^\frac{1}{2}}{\Lambdab(\ub)^\frac{1}{4}}\Lb\partial_x\Phi(t,x)\|^2_{L^2(\Sigma_\tau)}d\tau\Big)^\frac{1}{2}\\
&\lesssim \Big(\int\!\!\!\int_{\D_{t}} \frac{\Lambda(u)|\Lb\Phi|^2}{\Lambdab(\ub)^\frac{1}{2}}\Big)^\frac{1}{2}+\Big(\int\!\!\!\int_{\D_{t}} \frac{\Lambda(u)|\Lb\partial_x\Phi|^2}{\Lambdab(\ub)^\frac{1}{2}}\Big)^\frac{1}{2}.
\end{align*}
Both of the terms are in the same form as $\mathbf{I_2}$. Thus they can be bounded in the same manner. We then have
\begin{align*}
\mathbf{I}_4 \lesssim \varepsilon.
\end{align*}
Finally, in Case 2, we still have
\begin{align*}
\mathbf{I} \lesssim \varepsilon^3.
\end{align*}

\bigskip

By putting all the estimates together in \eqref{energy estimates}, for some universal constant $C_4$, we obtain that for all $t\in [0,T^*]$, 
\begin{equation*}
 \E(t)+\F(t)\leq 2C_0C_1\varepsilon^2+C_4\varepsilon^3.
\end{equation*}
We then take $\varepsilon_0$ such that 
\begin{equation*}
\varepsilon_0\leq \frac{2C_0C_1}{C_4}.
\end{equation*}
Therefore, for $\varepsilon\leq \varepsilon_0$ and for all $t\in [0,T^*]$, we have
\begin{align*}
 \E(t)+\F(t)&\leq 2C_0C_1\varepsilon^2+C_4\varepsilon^2 \times \frac{2C_0C_1}{C_4}\\
 &\leq 4C_0C_1\varepsilon^2.
\end{align*}
This proves the improved estimate \eqref{bootstrap assumption improved}.

This completes the proof of the main theorem.

\end{document}